\documentclass[reqno]{amsart}

\usepackage[utf8]{inputenc}
\usepackage{hyperref}
\usepackage{graphicx} 
\usepackage{mathrsfs, amsmath, amsthm, tikz-cd, color}
\usetikzlibrary{shadows}
\usepackage{scalerel,stackengine}
\usepackage{cleveref}

\numberwithin{equation}{section}

\stackMath
\newcommand\reallywidehat[1]{%
\savestack{\tmpbox}{\stretchto{%
  \scaleto{%
    \scalerel*[\widthof{\ensuremath{#1}}]{\kern-.6pt\bigwedge\kern-.6pt}%
    {\rule[-\textheight/2]{1ex}{\textheight}}
  }{\textheight}%
}{0.5ex}}%
\stackon[1pt]{#1}{\tmpbox}%
}
\usepackage{csquotes}       
\usepackage{microtype}      
\usepackage{amssymb}        
\usepackage{mathtools}      
\usepackage{xcolor}
\usepackage{eucal}
\usepackage[foot]{amsaddr}

\newtheorem{theorem}{Theorem}[section]
\newtheorem{proposition}[theorem]{Proposition}

\newtheorem*{lemma*}{Lemma}

\theoremstyle{definition}

\newtheorem*{definition*}{Definition}
\newtheorem*{proposition*}{Proposition}
\newtheorem{example}{Example}
\newtheorem{remark}[theorem]{Remark}

\newcommand{\R}{\mathbb{R}}
\newcommand{\Z}{\mathbb{Z}}

\renewcommand{\epsilon}{\varepsilon}

\newcommand{\tr}{\mathrm{tr}}

\newcommand{\im}{\mathrm{im \;}}

\newcommand{\Ac}{\mathcal{A}}
\newcommand{\Ic}{\mathcal{I}}
\newcommand{\Uc}{\mathcal{U}}

\newcommand{\Sc}{\mathcal{S}}
\newcommand{\Hc}{\mathcal{H}}

\newcommand{\dom}{\mathrm{dom}}

\newsavebox{\boxedtikzcdbox}

\title[Approximation tool for mixed- and equidimensional modeling]{An abstract approximation tool for mixed-dimensional and equidimensional modeling}

\author{Daniel F. Holmen\textsuperscript{1,*}}
\author{Jan M. Nordbotten\textsuperscript{1}}
\address{*Corresponding author: daniel.holmen@uib.no}
\address{\textsuperscript{1} Center for Modeling of Coupled Subsurface Dynamics, Department of Mathematics, University of Bergen, Allégaten 41, Bergen, Norway
}
\author{Jon E. Vatne\textsuperscript{2}}
\address{\textsuperscript{2} Department of Economics, BI Norwegian Business School, Kong Christian Frederiks plass 5, Bergen, Norway}

\date{}

\begin{document}
\begin{abstract}
Many coupled problems in engineering and science can be described by elliptic partial differential equations on adjacent domains, where the coupling can be considered either as a thin equidimensional overlap between the model domains, or as a lower-dimensional interface. Thereby we distinguish equidimensional and mixed-dimensional models of the same system, and the relationship between these modeling approaches is of natural interest.

In this paper, we construct an overlapping open cover for a class of simplicial geometries and construct a bounded cochain map from the simplicial de Rham complex to the \v{C}ech-de Rham complex associated with the overlapping cover. Thus, we establish an isomorphism between simplicial de Rham complexes (i.e. functions and forms on mixed-dimensional partitions and their differentials) and subcomplexes of \v{C}ech-de Rham complexes (i.e. functions and forms on equidimensional partitions and their differentials), which serves as an abstract approximation tool for comparing mixed-dimensional problems to the equidimensional version of the same problem.
\end{abstract}
\maketitle


\section{Introduction}\label{section: 1}
Field models of continuum mechanics are almost invariably realized as partial differential equations \cite{truesdell2004non, temam2005mathematical}. When different physical objects are coupled, the resulting models are then naturally realized as a system of coupled partial differential equations, wherein the coupling may either be across an interface (such as for fluid-structure-interaction \cite{richter2017fluid} and contact mechanics \cite{kikuchi1988contact}), or wherein the domains fully or partially overlap (e.g. Biot’s theory of poroelasticity \cite{coussy2004poromechanics} and peridynamics \cite{oterkus2021peridynamic}).

Many of these field models have a structure that can be described in terms of differential complexes. The most famous of these is the de Rham complex, which unifies the treatment of the classical differential operators in vector calculus: the gradient, curl and divergence.  When equipped with inner-product spaces, it gives rise to the scalar and vector Laplace equations, thus linking strongly the partial differential equations modeling processes as disparate as diffusion, heat transfer, electromagnetism and so forth \cite{arnoldFEEC2,arnoldFEEC3, deschamps1981electromagnetics}. However, it is known that also other field theories in continuum mechanics have an associated differential complex, such as Stokes equations for fluids \cite{falkstokes} and the equations of linear elasticity \cite{AFWelasticity, pauly2023elasticity}. 

We herein refer to coupled problems where the coupling is realized as interactions across interfaces separating non-overlapping domains as mixed-dimensional. In contrast, we formalize coupled problems where the coupling is realized as a point interaction between overlapping domains as equidimensional. If an equidimensional coupled problem has a “thin” overlap, we characterize the overlap in terms of a length scale $\epsilon$.  It is of both mathematical and practical interest to answer the question of whether the limit of $\epsilon \to 0$ will indeed be a mixed-dimensional model, and if yes, what the approximation properties of the mixed-dimensional ($\epsilon=0$) model will be relative to the equidimensional ($\epsilon>0$) model. Indeed, such questions have been raised and answered for many particular instances of coupled problems on thin domains. Notable examples include fluid flow in thin porous layers \cite{yortsos1995theoretical, armiti2019existence, list2020rigorous}, mechanics of brittle fracture \cite{chambolle2004approximation, linse2017convergence}, flow networks embedded in a porous material \cite{d2008coupling, koppl2018mathematical, zuninolaurino2019derivation}.

Recent work has established that various spatially coupled problems can also be cast in the context of differential complexes, but now with the more general notion of double complexes, wherein two anti-commuting differential operators are considered. For equidimensional coupled problems we consider the \v{C}ech-de Rham complex, where the first differential operator is a difference operator acting on an overlapping open cover, while the second differential operator is the standard exterior derivative of the de Rham complex. When considering the \v{C}ech-de Rham complex with square-integrable coefficients and equipped with an inner product, we get a well-posed Hodge-Laplace problem which governs certain equidimensional coupled problems \cite{cdrhodge}. 

For coupled problems of the mixed-dimensional variety, a different cochain complex than the \v{C}ech-de Rham complex is needed. Various approaches have recently been developed to address the development of mixed-dimensional physical models within porous media \cite{mdG, boon2023mixed} and electromagnetism \cite{brauer1993mixed, buffa2003electric}, as well as for the construction and analysis of numerical methods \cite{licht2017complexes, falk2015double}. Common to these constructions is that they essentially consider (extensions of) a double complex that can be identified as the simplicial de Rham complex, wherein the first differential operator is the difference between higher-dimensional traces, while the second differential operator is again the standard differential operator of the de Rham complex, restricted to the corresponding manifolds. Another interpretation of the simplicial de Rham complex is that it represents the limit as the width of the overlaps in the \v{C}ech-de Rham complex goes towards zero.

Due to the increasing interest in the application of double complexes in the analysis and design of methods for addressing coupled problems in applied sciences, it is timely to ask whether there can be established an explicit connection between the simplicial de Rham complex and the \v{C}ech-de Rham complex. Such a connection, realized in the form of a bounded cochain map, is established herein. 

Our main result is therefore to realize a bounded cochain map from the Hilbert simplicial de Rham complex into the Hilbert \v{C}ech-de Rham complex. More precisely, we state a locally (i.e. component-wise for direct summands) defined linear mapping, and establish two key properties: commutativity with the respective differential operators, and boundedness in the graph norm. The cochain map we describe defines an isomorphism between the simplicial de Rham complex and a subcomplex of the \v{C}ech-de Rham complex, which serves as a tool for comparing mixed-dimensional models with their equidimensional counterparts.

The next section is dedicated to introducing the reader to cochain complexes and the aforementioned relevant examples of double complexes. In \Cref{section: cochain map}, we describe the construction of a cochain map between the two complexes after a preliminary example. \Cref{section: properties of the cochain map} shows that the cochain map is bounded from above and below, and therefore defines an isomorphism between the domain complex of the simplicial de Rham complex and its image (realized as a subcomplex of the \v{C}ech-de Rham complex). Moreover, we introduce weights in the inner products of the simplicial de Rham complex. The weighted inner products establish invariant boundedness with respect to the thickness of the overlaps in the open cover. 

\section{Mathematical background}
In this section we briefly review the fundamental properties of cochain complexes and their weak generalization, which is referred to as Hilbert complexes.  Moreover, we explicitly recall the de Rham complex, and the two double complexes needed in this work, the \v{C}ech-de Rham complex and the simplicial de Rham complex. 

\subsection{Cochain complexes}\label{sub: cochain complexes}
A \emph{cochain complex} is a sequence of objects, for example modules, abelian groups or vector spaces, with homomorphisms, e.g. linear maps,
\begin{equation} \label{eq. cochain complex}
... \xrightarrow[]{d^{k-1}} C^k \xrightarrow[]{d^k} C^{k+1} \xrightarrow[]{} ... 
\end{equation}
with the property $d^{k} d^{k-1} = 0$ for all $k$, or equivalently that $\im d^{k-1} \subset \ker d^k$. We call the homomorphisms $d^k$ the \emph{differentials} of the cochain complex, since the differentials in the de Rham complex correspond to the differential operators gradient, curl and divergence (see \Cref{sub: de rham}). A cochain complex is exact at $C^k$ if $\im d^{k-1} = \ker d^k$, and we say that a cochain complex is exact if it is exact everywhere. The extent to which the cochain complex fails to be exact at $C^k$ is measured by the $k$-th cohomology space, which is the quotient space $\mathcal{H}^k(C) = \ker d^k / \im d^{k-1}$. Indices of differential operators will frequently be omitted whenever convenient, and we write $(C^\bullet, d)$ as shorthand for \cref{eq. cochain complex}.

A \emph{double (cochain) complex} is an array of objects $C^{p,q}$ with $p,q \in \Z$, endowed with two differentials: a horizontal differential $d_h: C^{p,q} \to C^{p+1, q}$ and a vertical differential $d_v: C^{p,q} \to C^{p, q+1}$, such that
\begin{subequations}
\begin{align}    
d_h d_h &= 0, \label{eq. horizontal diff} \\ 
d_v d_v &= 0, \label{eq. vertical diff} \\ 
d_v d_h + d_h d_v &= 0. \label{eq. anti-commute}
\end{align}  
\end{subequations}
That is, each row $C^{\bullet, q}$ and each column $C^{p, \bullet}$ is a cochain complex, and each rectangle in the double complex anticommutes:
\begin{equation}
\begin{tikzcd}
              & ...                                         & ...                                    &     \\
... \arrow[r] & {C^{p,q+1}} \arrow[u] \arrow[r, "d_h"]      & {C^{p+1,q+1}} \arrow[r] \arrow[u]      & ... \\
... \arrow[r] & {C^{p,q}} \arrow[u, "d_v"] \arrow[r, "d_h"] & {C^{p+1,q}} \arrow[r] \arrow[u, "d_v"] & ... \\
              & ... \arrow[u]                               & ... \arrow[u]                          &    
\end{tikzcd}
\end{equation}
The \emph{total complex} of a double complex $C^{\bullet, \bullet}$ is a cochain complex defined by taking the direct sum of the anti-diagonals, $T(C)^k = \bigoplus_{p+q=k} C^{p,q}$. We define the differential of the total complex by adding the vertical and horizontal differentials: $D = d_v + d_h$. Since the two differential operators $d_v$ and $d_h$ anti-commute\footnote{Anti-commuting differential operators is one of the two conventions, used in e.g. \cite{aluffialgebra}. A different convention (e.g. in \cite{bott-tu}) is to require the two differential operators $d_v$ and $d_h$ to commute, and define the total differential as $D = d_v + (-1)^k d_h$.}, the total differential $D$ satisfies $D^{k+1} D^{k} = 0$. As a convention, we will write $C^k$ and $C^{p,q}$ for the total cochain complex and the double cochain complex, respectively. 

Given two cochain complexes $(A^\bullet, d_A)$ and $(B^\bullet, d_B)$, a \emph{cochain map} is a collection of homomorphisms $f^k: A^k \to B^k$ such that for each $k$, the following diagram commutes:
\begin{equation}
\begin{tikzcd}
A^k \arrow[r, "d_A"] \arrow[d, "f^k"] & A^{k+1} \arrow[d, "f^{k+1}"] \\
B^k \arrow[r, "d_B"]                  & B^{k+1}                 
\end{tikzcd}
\end{equation}
A cochain map induces a linear map in cohomology: $f_*^\bullet: \Hc^\bullet(A) \to \Hc^\bullet(B)$.

We are ultimately interested in cochain maps between double complexes.
Such maps are given by bigraded components for each bidegree $(p,q)$ that commute simultaneously with both the horizontal differential $\delta$ and the vertical differential $d$. Such a bigraded cochain map means that each of the sides in the following diagram commutes (whereas the top and bottom squares anti-commute):
\begin{equation}    
\begin{tikzcd} [sep = .5 cm]
& & & A^{p,q+1} \arrow [rr, "\delta_A" {description, left}] \arrow [dd] & & A^{p+1,q+1} \arrow [dd]  \\
& & A^{p,q} \arrow [ru, "d_A"] \arrow [rr, "\delta_A" {description, right}, crossing over]  & & A^{p+1,q} \arrow [ru, "d_A"] \\
& & & B^{p,q+1} \arrow [rr, "\delta_B" {description, left}] & & B^{p+1,q+1} \\
 & & B^{p,q} \arrow [ru, "d_B"] \arrow [rr, "\delta_B" {description, right}] \arrow [from = uu, crossing over] & & B^{p+1,q} \arrow [ru, "d_B"] \arrow [from = uu, crossing over]\\
\end{tikzcd}
\end{equation}

By forming direct summands of bigraded components, a cochain map between double complexes induces a \emph{total cochain map} $f^k: T(A)^k \to T(B)^k$ satisfying $f^{k+1} D_A = D_B f^k$.
All maps between total complexes considered here will appear in this way, though there are also other cochain maps between total complexes, where only the differential of the total complex (but not the individual vertical and horizontal of the double complex) commute.
\subsection{Hilbert complexes} \label{sub: hilbert complexes}
We continue by briefly summarizing some key definitions of Hilbert complexes, as this generalization of the construction is needed for analysis. For a more detailed exposition to Hilbert complexes, we refer to \cite{arnoldFEEC2}.

A \emph{Hilbert complex} is a cochain complex where each $C^k$ is a Hilbert space and the differential operators $d^k$ are closed densely defined unbounded linear operators. Since the differential operators are densely defined but not necessarily defined on the entire $C^k$, we refer to both the full Hilbert complex and the subcomplex defined by the subspaces $\dom \; d^k \subset C^k$, known as the \emph{domain complex}. Moreover, we will be working with \emph{closed Hilbert complexes}, meaning that each differential operator has closed range. 

Each of the spaces in the cochain complex are equipped with inner products, which defines the adjoint of the differential operator. For $a \in C^k$, $b \in C^{k+1}$, the adjoint is given by the following equation:
\begin{equation}
\langle d^* b, a \rangle_{C^k} = \langle b, da \rangle_{C^{k+1}}. 
\end{equation}
The adjoint of the differential is a degree $-1$ operator and it is called the \emph{codifferential}. Since it also satisfies $d^*_k d^*_{k+1} =0$, the codifferential defines a cochain complex called the \emph{adjoint complex}, which goes in the opposite direction to the original cochain complex. Also associated to a Hilbert complex is the \emph{Hodge-Laplacian}, defined as $\Delta_d^k = d^{k-1} d_k^* + d^*_{k+1} d^k$. Hilbert complexes admit an orthogonal decomposition called a \emph{Hodge decomposition}:
\begin{equation}
C^k = \im d^{k-1} \oplus \ker \Delta_d^k \oplus \im d_{k+1}^*.
\end{equation}

Given a bigraded Hilbert complex $(C^{\bullet, \bullet}, d_v, d_h)$, we can define the inner product on the total complex by taking the sum of the inner products along the anti-diagonal:
\begin{equation}
\langle a, b \rangle_{C^k} = \sum_{p+q=k}  \langle a_{p,q}, b_{p,q} \rangle_{C^{p,q}}. \label{eq: double inner product}
\end{equation}
If $\dom(d_v) \cap \dom(d_h)$ is dense and $D = d_v + d_h$ is a closed operator, then the total complex $(C^\bullet, D)$ of the double Hilbert complex is therefore also a closed Hilbert complex with the inner product given by \cref{eq: double inner product}. We can therefore define the Hodge-Laplacian on the total Hilbert complex, and the total complex admits a Hodge decomposition.
\subsection{The de Rham complex}\label{sub: de rham}
The classical ellpitic partial differential equations based on operators such as gradient, divergence and curl are understood to be expressions of the fundamental complex named after de Rham. We briefly review its construction below. 

Given a smooth $n$-dimensional manifold $\Omega$, a differential form of degree one is a smooth section of the cotangent bundle, i.e. an element of $T^*\Omega \otimes C^\infty(\Omega)$. More generally, a differential form of degree $k$ is an element of $\mathrm{Alt}^k(T^*\Omega) \otimes C^\infty(\Omega)$, where $\mathrm{Alt}^k(T^*\Omega)$ denotes the $k$-th exterior power of $T^*\Omega$. We denote the spaces of differential $k$-forms by $C^\infty \Lambda^k(\Omega)$. If $\{dx_j\}_{j=1}^n$ is basis for the cotangent bundle $T^* \Omega$, then the basis for $C^\infty \Lambda^k(\Omega)$ is given by $dx_{i_1} \wedge ... \wedge dx_{i_k}$, and for a $k$-form $\alpha$ we write 
\begin{equation}
\alpha = \sum_{i_1 < ... < i_k} a_i \; dx_{i_1} \wedge ... \wedge dx_{i_k}  =  \sum_{i \in \mathcal{N}^k} a_i \; dx_{i}, \qquad a_i \in C^\infty(\Omega). \label{eq: 2.11}
\end{equation}
Here $\mathcal{N}^k$ is the set of k-tuples with strictly increasing entries, as implied by \cref{eq: 2.11}. The \emph{exterior derivative} of a $k$-form $\alpha$ is the anti-symmetric part of the directional derivative, defined as follows:
\begin{equation}
 d\alpha = \sum_{j=1}^n \sum_{i \in \mathcal{N}^k} \frac{\partial a_i}{\partial x_j} dx_j \wedge dx_i.    
\end{equation}
The exterior derivative is a linear map $d^k: C^\infty \Lambda^k(\Omega) \to C^\infty \Lambda^{k+1}(\Omega)$ and satisfies $d^{k+1} d^k = 0$. The graded space of differential forms together with the exterior derivative defines the \emph{de Rham complex} $(C^\infty\Lambda^\bullet, d)$. A differential form of degree $n$ is called a \emph{volume form}. A non-vanishing volume form $\omega \in C^\infty \Lambda^n(\Omega)$ determines a unique orientation of $\Omega$ for which $\omega$ is positively oriented at each $x \in \Omega$. All manifolds we work with are orientable, and we denote a choice of volume form on $\Omega$ by $\mathrm{vol}_\Omega$.

By considering differential forms with $L^2$-coefficients, we can define a Hilbert complex with an inner product: 
\begin{align} \label{eq: inner product}
\langle \alpha, \beta \rangle_{\Lambda^k(\Omega)} &= \int_{\Omega} \sum_{j \in \mathcal{N}^k} \alpha_j \beta_j \; \mathrm{vol}_\Omega, \qquad \alpha_j, \beta_j \in L^2(\Omega).
\end{align}
We refer to the Hilbert complex $(L^2\Lambda^\bullet(\Omega), d)$ as the $L^2$ de Rham complex. The exterior derivative is densely defined but not defined everywhere on $L^2\Lambda^k$, so we can instead consider the domain complex: 
\begin{equation}
H\Lambda^k(\Omega) := \{\alpha \in L^2\Lambda^k(\Omega) : d\alpha \in L^2\Lambda^{k+1}(\Omega)\}.
\end{equation}
Equivalently, we can define the domain complex as the closure of $C^\infty \Lambda^\bullet$ with respect to the induced norm. 

In summary, we have three different de Rham complexes: the de Rham complex with smooth coefficients $(C^\infty \Lambda^\bullet(\Omega), d)$, the de Rham complex with square-integrable coefficients $(L^2\Lambda^\bullet(\Omega), d)$ and its associated domain complex $(H\Lambda^\bullet(\Omega), d)$.

\subsection{The \v{C}ech-de Rham complex}
The structure of open covers and their overlap is given by the complex named after \v{C}ech. Combining the \v{C}ech complex with the de Rham complex gives a framework for studying multi-physics problems on (partially) overlapping domains \cite{cdrhodge}, which motivates identifying the \v{C}ech-de Rham complex as the underlying structure for coupled problems.

We continue by describing the double complex known as the \v{C}ech-de Rham complex, which is obtained by consider the \v{C}ech complex with values in $C^\infty\Lambda^\bullet$. For more details including results regarding the cohomology of the complex, we refer to \cite{bott-tu}. The Hilbert complex version of the \v{C}ech-de Rham complex is defined and related to modeling of coupled problems in \cite{cdrhodge}.

The \v{C}ech-de Rham complex is a double complex where one of the differential operators is the exterior derivative, the other is an operator taking differences, or more generally, alternating sums restricted to overlaps of open sets. 

Let $\Uc = \{U_i\}_{i \in \mathcal{I}}$ be an open cover of a smooth manifold $\Omega \subset \R^n$. We will throughout this text assume to be working with a (locally) finite open cover, although the definitions below hold for a countably infinite cover. 

We define the $p$-cochains of the \v{C}ech complex with values in the space of differential $q$-forms $C^\infty\Lambda^q$ as follows:
\begin{equation}
 C^\infty \Ac^{p,q} = \prod_{i_0 < ... < i_p}  C^\infty\Lambda^q(U_{i_0, ..., i_p}).
\end{equation}
Here, $U_{i_0, ..., i_p}$ is short-hand notation for the intersection $U_{i_0} \cap ... \cap U_{i_p}$. Moreover, we will routinely write $i$ for a multi-index $(i_0, ..., i_p)$, and denote the set of strictly increasing multi-indices of length $p+1$ by $\Ic^p$. Note that in \cref{eq: 2.11}, a multi-index $i \in \mathcal{N}^k$ is of length $k$, whereas a multi-index $i \in \Ic^p$ has length $p+1$. This discrepancy can be explained by the fact that $\mathcal{N}^k$ corresponds to differential forms of degree $k$ and $\Ic^p$ corresponds to cochains of degree $p$ in $C^\infty \Ac^{p,q}$. Moreover, given a multi-index $i \in \Ic^p$, we introduce the subset $\Ic_i$ which are all multi-indices containing $i$. We may also specify with a superscript: $\Ic_i^{p+r} \subset \Ic^{p+r}$ are all multi-indices of length $p+r+1$ which contains $i$.

We define a difference operator by taking alternating sums on the degree $p+1$ overlaps:
\begin{align}
(\delta \alpha)_i &=  \sum_{l=0}^{p+1} (-1)^{k+l}  \; \alpha_{i_0, ..., \hat{i}_l, ..., i_{p+1}}|_{U_i}, &
\forall i \in \mathcal{I}^{p+1}.
\end{align}
Here, $k=p+q$ denotes the total degree and the hat in $\hat{i}_l$ denotes that this index is omitted. When accounting for each $i \in \Ic^{p+1}$, we get a difference operator $\delta^p: C^\infty \Ac^{p,q} \to  C^\infty \Ac^{p+1,q}$.

The difference operator $\delta$ satisfies the relation $\delta^{p+1} \delta^p = 0$, and we also have the exterior derivative $d$ acting on the degree of differential forms. By letting the difference operator $\delta$ alternate with the degree $k$, we have two differential operators that anti-commute. We therefore have a double cochain complex known as the \emph{double-graded \v{C}ech-de Rham complex}.

The \v{C}ech-de Rham complex admits all the properties of a double complex described in \cref{sub: cochain complexes}, and we can therefore construct a total \v{C}ech-de Rham complex. Moreover, following the theory from \cref{sub: hilbert complexes}, we consider the \v{C}ech-de Rham complex with coefficients in $H\Lambda^\bullet$. That is, 
\begin{equation} \label{eq: domain complex cdr}
H\Ac^{p,q} = \prod_{i \in \Ic^p} H\Lambda^q(U_i), \qquad H\Ac^k = \bigoplus_{p+q=k} H\Ac^{p,q}.
\end{equation}
In future sections we will omit the $H$ as a prefix and simply write $\Ac^{p,q}$ instead of $H\Ac^{p,q}$.

\subsection{The simplicial de Rham complex}\label{sub: simplicial de rham}
The subdivision of a domain into non-overlapping open sets and their boundaries gives rise to the simplicial complex. As in the case of the \v{C}ech-de Rham complex, the combination of the simplicial complex with the de Rham complex forms the foundation for considering interface-coupled problems, and more generally also mixed-dimensional problems. For a more general construction and more details, we refer to \cite{mdG}.

A simplicial complex is called \emph{pure} if all its facets are of the same dimension $n$, meaning that all lower-dimensional $(n-p)$-simplices are part of the boundary of a simplex of dimension $(n-p+1)$. We consider a pure embedded simplicial complex where the $n$-simplices $\{\Omega^n_i\}_{i \in \Ic}$ are indexed by the same index set $\Ic$ as the open cover for the \v{C}ech-de Rham complex, hence the index set is assumed to be finite. Moreover, we assume a bijection between the $(n-p)$-simplices $\{\Omega^{n-p}_j\}_{j \in \Ic^p}$ and the $p$-overlaps from the open cover $\{U_j\}_{j \in \Ic^p}$. To account for this 1-1 correspondence, we ignore the lower-dimensional outer simplices, as they don’t correspond to an overlap of domains.

For each $\Omega_{j} \subset \partial \Omega_i$, we consider the inclusion map of a $p$-simplex $\Omega_{j}$ which is the boundary of a $p+1$-simplex $\Omega_i$. The pullback of the inclusion map $\Omega_{j} \to \Omega_i$ is a restriction of differential forms $C^\infty \Lambda^\bullet(\Omega_i) \to C^\infty \Lambda^\bullet(\Omega_{j})$. In the case where we consider distributional differential forms in $H\Lambda^\bullet$, the induced pullback is the trace operator.

We define a boundary operator called the \emph{jump operator}, acting on the simplicial de Rham complex:
\begin{align}\label{eq: jump operator}
(\delta_S \alpha)_i &=  \sum_{l=0}^{p+1} (-1)^{k+l} \tr_i \; \alpha_{i_0, ..., \hat{i}_l, ..., i_{p+1}}, &
\forall i \in \mathcal{I}^{p+1},
\end{align}
where the $\hat{i}_l$ is again understood to be index omitted, and the trace is understood to be the trace from $\Omega_{i_0, ..., \hat{i}_l, ..., i_{p+1}}$ to its boundary $\Omega_i$. Similarly to the \v{C}ech-de Rham complex, the jump operator and the exterior derivative form a double complex when acting on differential forms on the simplicial complex. Once again, we are ultimately interested in the Hilbert complex version of the double complex described above.

We introduce subspaces such that the image of the trace operator $\tr_j$ lies in $H\Lambda^q(\Omega_j)$ rather than just in $L^2\Lambda^q(\Omega_j)$, which leads to the recursive definition of the following subcomplex:
\begin{equation}
H\Lambda^k(\Omega_i, \tr) = \{\alpha \in H\Lambda^k(\Omega_i) : \tr_{j} \; \alpha \in H\Lambda^k(\Omega_{j}, \tr), \forall j \in \Ic_i \}. \label{eq: regular trace}
\end{equation}

This subcomplex allows us to take iterative traces of differential forms without losing regularity. We simply write $\tr^p$ for $p$-times iterated trace, whenever the domain and codomain are understood from the context.

We define the \emph{double-graded simplicial de Rham complex} to be the product of the subspaces from \cref{eq: regular trace}:
\begin{equation}\label{eq: domain complex Simplicial de Rham}
H\Sc^{p,q} = \prod_{i \in \Ic^p} H\Lambda^q(\Omega_{i}, \tr).
\end{equation}

The exterior derivative is acting on the degree of the differential forms:
\begin{equation}
d_\Sc^q: H\Sc^{p,q} \to H\Sc^{p, q+1}.
\end{equation}

When we consider the operator in \cref{eq: jump operator} acting on each domain $\Omega_i$ for each multi-index $i \in \Ic^{p+1}$, we get the \emph{jump operator}:
\begin{equation}
\delta_\Sc^p: H\Sc^{p,q} \to H\Sc^{p+1, q}.
\end{equation}

We define the \emph{total simplicial de Rham complex} to be the total complex of the aforementioned double complex:
\begin{equation}
H\Sc^k = \bigoplus_{p+q=k} H\Sc^{p,q}.
\end{equation}
Again, we will omit the prefix $H$ in future sections.


There are a few limitations to the embedding we describe in \cref{section: cochain map} which are worth highlighting. The most significant limitation is that we do not account for mixed-dimensional geometries where lower-dimensional features do not extend towards the outer boundary. A second limitation is that we do not consider geometries where the index sets are not matching, e.g. higher degree of intersections. In \cref{fig: permissable and nonpermissable}, we present two geometries which falls under these two limitations, and two geometries which are allowed, despite being similar to their non-permissible counterpart. We do not believe either of the restrictions mentioned here are indispensable but imposing them simplifies the exposition considerably.

\begin{figure}[htb]
    \centering
   \includegraphics[scale=0.65]{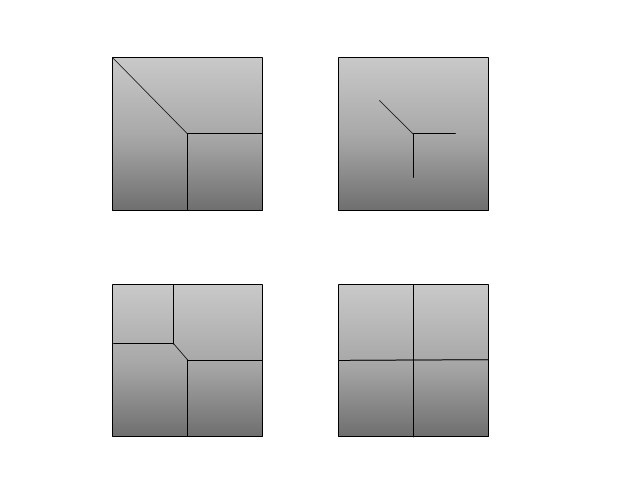}
    \caption{Top left: permissible.
    Top right: not permissible, since the lower-dimensional faces don't extend to the boundary.
    Bottom left: permissible. Bottom Right: not permissible, since the mixed-dimensional and equidimensional indexing don't match.
    }
    \label{fig: permissable and nonpermissable}
\end{figure}

\subsection{Differences between the \v{C}ech- and simplicial de Rham complexes}
The two cochain complexes we consider are similar in several ways. They both consist of differential forms in a product space with a hierarchy of codimension/degree of overlap. Under our stated assumptions, we can index the hierarchies by the same index set $\Ic$. The 1-1 correspondence also is assumed to hold for the corresponding multi-indices. Both cochain complexes have the exterior derivative as one of the differential operators, and the jump/difference operator increases the degree $p$ of codimension/degree of overlap.  There is also a clear conceptual similarity between the two complexes which motivates the comparison; they are both related to spatially coupled problems.

One of the differences between the cochain complexes is their lengths. If the maximal degree of a differential form on a submanifold exceeds the dimension of the submanifold, the differential form is evaluated to zero. Therefore, non-trivial differential forms on the simplicial de Rham complex has maximal degree $n-p$, where $p$ is the codimension of the simplex it is defined on. On the other hand, the \v{C}ech-de Rham complex is equidimensional, and we can therefore have non-trivial differential forms of degree $n$ on intersections of degree $p$.

As a consequence, the non-zero entries of the double complexes are not matching. The \v{C}ech-de Rham complex is defined for each $\Ac^{p,q}$ with $p \in \{0, ..., p_{max}\}$ and $q \in \{0, ..., n\}$, where $p_{max}$ denotes the maximal number of overlaps which satisfies $p_{max} \leq n$ by previous assumptions on the index set $\Ic$. 

On the other hand, the simplicial de Rham complex is only defined for $\Sc^{p,q}$ satisfying $0 \leq p+q \leq n$. Thus, the non-zero entries in the bigraded simplicial de Rham complex form a triangular shape, while the non-zero entries of the bigraded \v{C}ech-de Rham form a rectangular in shape. The total complex of the \v{C}ech-de Rham complex is of length $n + p_{max}$, but the the total simplicial de Rham complex has always length $n$:
\begin{equation}
\begin{tikzcd}
	0 & {\mathcal{S}^0} & {...} & {\mathcal{S}^n} & {0,} \\
	0 & {\mathcal{A}^0} & {...} & {\mathcal{A}^n} & {...} & {\mathcal{A}^{n+p_{max}}} & 0
	\arrow[from=1-1, to=1-2]
	\arrow[from=1-2, to=1-3]
	\arrow[from=1-3, to=1-4]
	\arrow[from=1-4, to=1-5]
	\arrow[from=2-1, to=2-2]
	\arrow[from=2-2, to=2-3]
	\arrow[from=2-3, to=2-4]
	\arrow[from=2-4, to=2-5]
	\arrow[from=2-5, to=2-6]
	\arrow[from=2-6, to=2-7]
\end{tikzcd}
\end{equation}
If we naively construct a map $\Xi^\bullet$ from the $\Sc^\bullet$ to $\Ac^\bullet$, we see that the differential $D^n$ maps everything in $\Sc^n$ to zero. Thus the image of $\Xi^n$ needs to lie in the kernel $\ker(D^n) \subset \Ac^n$ in order for us to have a cochain map. Equivalently, we can say that we are constructing a cochain map to the \emph{truncated} \v{C}ech-de Rham complex, which is the subcomplex 
\begin{equation}
0 \to \Ac^0 \to ... \to \Ac^{n-1} \to \ker(D^n) \to 0.
\end{equation}
Since we will in this work exclusively consider cochain maps mapping $\Sc^n$ to $\ker(D^n) \subset \Ac^n$, the distinction between the truncated and full \v{C}ech-de Rham complex will not be of great importance. Consequently, we will abuse the notation slightly, and write $\Ac^n$ for the last part of the truncated complex and simply refer to the result as the \v{C}ech-de Rham complex.
\section{Constructing cochain maps} \label{section: cochain map}
We now address the main goal of this paper, which is to construct an injective bigraded cochain map from the simplicial de Rham complex to the \v{C}ech-de Rham complex. More specifically, we are considering the Hilbert complex version of the respective complexes, as defined in \cref{eq: domain complex Simplicial de Rham} and \cref{eq: domain complex cdr}. Note that since we are constructing a cochain map between the domain complexes, the differential operators (in particular the exterior derivative) is now globally defined. 

Although we are considering the simplicial de Rham complex, a larger class of mixed-dimensional geometries can be represented by bijectively mapping the more general mixed-dimensional geometry to the simplicial geometry.

\subsection{Notation}
We let $\{\Omega_i\}_{i \in \Ic}$ denote a given simplicial complex and let $\Uc = \{U_i\}_{i \in \Ic}$ be an associated open cover, indexed by the same set $\Ic$. The construction of such open covers is described later in \cref{subsection: constructing open cover}. The simplicial de Rham complex is written as $(\Sc^{\bullet, \bullet}, d_\Sc, \delta_\Sc)$, and the \v{C}ech-de Rham complex is written as $(\Ac^{\bullet, \bullet}, d_\Ac, \delta_\Ac)$.

We write $\Xi$ for the cochain map between the double complexes, with components $\Xi^{p,q}: \Sc^{p,q} \to \Ac^{p,q}$ for ${p,q \geq 0}$.

\begin{equation}    
\begin{tikzcd} [sep = .5 cm]
& & & \Sc^{p,q+1} \arrow [rr, "\delta_\Sc" {description, left}] \arrow [dd] & & \Sc^{p+1,q+1} \arrow [dd]  \\
& & \Sc^{p,q} \arrow [ru, "d_\Sc"] \arrow [rr, "\delta_\Sc" {description, right}, crossing over]  & & \Sc^{p+1,q} \arrow [ru, "d_\Sc"] \\
& & & \Ac^{p,q+1} \arrow [rr, "\delta_\Ac" {description, left}] & & \Ac^{p+1,q+1} \\
 & & \Ac^{p,q} \arrow [ru, "d_\Ac"] \arrow [rr, "\delta_\Ac" {description, right}] \arrow [from = uu, crossing over] & & \Ac^{p+1,q} \arrow [ru, "d_\Ac"] \arrow [from = uu, crossing over]\\
\end{tikzcd}
\end{equation}

By assumption, for each $\Omega_i$ with $i \in \Ic$, there is a corresponding open set $U_i$. The intersection $U_i \cap U_j$ corresponds to $\Omega_{i,j}$, which is the common boundary of $\Omega_i$ and $\Omega_j$. We define $\Tilde{U}_i \subset U_i$ to be the part of $U_i$ which does not intersect with any of the other open sets $U_j$, $j \in \Ic$. More precisely, for $i \in \Ic^p$, we define $\Tilde{U}_i$ as follows:
\begin{equation}\label{eq: u tilde}
\Tilde{U}_i = U_i \setminus \{ \bigcup_{j \in \Ic^p, j \neq i} \overline{U_j} \}, \qquad i \in \Ic^p.
\end{equation}

For each $p \in \{0, ..., p_{max}\}$ and $i \in \Ic^p$, we consider a transformation $\phi_i: \Tilde{U}_i \to \Omega_i$. Each of the maps $\phi_i$ induces a pullback on differential forms: $\phi_i^*: H\Lambda^\bullet(\Omega_i, \tr) \to H\Lambda^\bullet(\Tilde{U}_i)$. Our goal now is to extend the map $\phi_i^*: H\Lambda^\bullet(\Omega_i, \tr) \to H\Lambda^\bullet(\Tilde{U}_i)$ to a map which has codomain $H\Lambda^\bullet({U}_i)$.

\subsection{Example in \texorpdfstring{$2D$}{2D}} \label{section: example}
We consider the 2-dimensional mixed-dimensional geometry consisting of three domains $\Omega_{0}, \Omega_1, \Omega_2$, three interfaces $\Omega_{0,1}, \Omega_{0,2}, \Omega_{1,2}$ and an intersecting point $\Omega_{0,1,2}$, as illustrated in \cref{fig: simplicial cech figure}.
\begin{figure}[htb]
    \centering
   \includegraphics[scale=0.5]{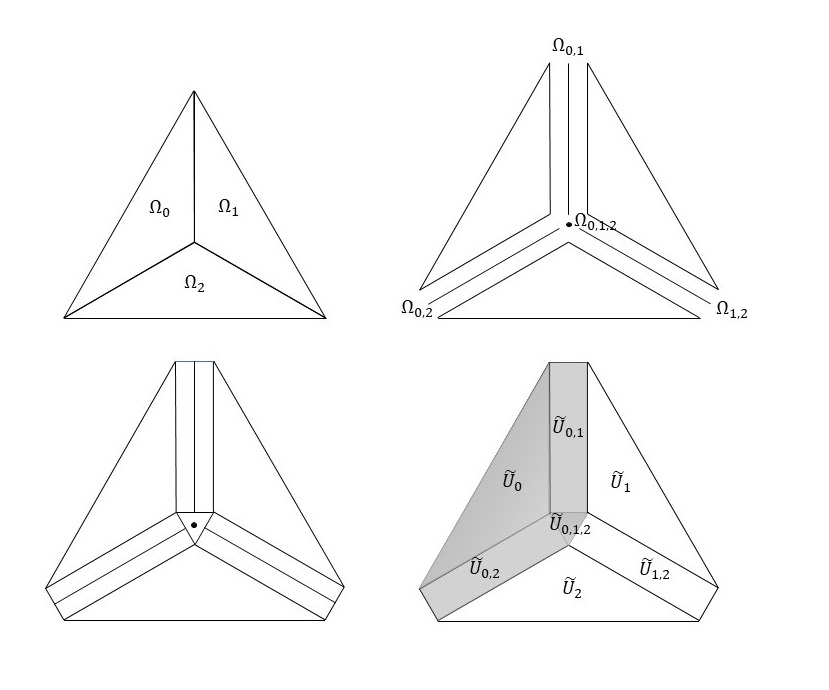}
    \caption{Top left: a simplicial geometry with three adjacent triangles. 
    Top right: the mixed-dimensional decomposition, consisting of the three triangles, the three pairwise shared edges and a vertex in the center.
    Bottom left: the edges are stretched out in their normal direction.
    Bottom right: the open set $U_0$ is highlighted in gray.
    }
    \label{fig: simplicial cech figure}
\end{figure}
The figure also depicts a possible way to construct an open cover for the given simplicial geometry.

For the example in question, we have the following diagram of cochain maps between the two complexes:
\begin{equation}
\begin{tikzcd}
0 \arrow[r] & \Sc^0 \arrow[r, "D_\Sc"] \arrow[d, "\Xi^0"] & \Sc^{1} \arrow[d, "\Xi^{1}"] \arrow[r, "D_{\Sc}"] & \Sc^{2} \arrow[d, "\Xi^2"] \arrow[r] & 0 \\
0 \arrow[r] & \Ac^0 \arrow[r, "D_\Ac"]                    & \Ac^{1} \arrow[r, "D_{\Ac}"]                      & \Ac^{2} \arrow[r]                    & 0
\end{tikzcd}
\end{equation}
That is, we need to construct cochain maps $\Xi^0$, $\Xi^1 = \Xi^{1,0} \oplus \Xi^{0,1}$ and $\Xi^2 = \Xi^{2,0} \oplus \Xi^{1,1} \oplus \Xi^{0,2}$. For $p=0$, we have 3 domains, for $p=1$ we also have 3 domains, and for $p=2$ we have 1 domain in both the mixed-dimensional and equidimensional case.
We write $f = (f_0, f_1, f_2)$ for functions $f_i$ with domains $\Omega_i$, $g = (g_{0,1}, g_{0,2}, g_{1,2})$ for functions $g_{j}$ with domains $\Omega_{j}$ (with double-indices $j = (j_0, j_1) \in \Ic^1$) and $c \in \R$ for a constant on $\Omega_{0,1,2}$. For differential forms with degree $q>0$, we write $\alpha = (\alpha_0, \alpha_1, \alpha_2)$ for differential 1-forms $\alpha_i$ with domains $\Omega_i$, $\beta = (\beta_{0,1}, \beta_{0,2}, \beta_{1,2})$ for differential 1-forms $\beta_{j}$ with domains $\Omega_{j}$ and $\omega = (\omega_0, \omega_1, \omega_2)$ for differential 2-forms $\omega_i$ with domains $\Omega_{i}$. 
We make use of the pullbacks $\phi^*_i: H\Lambda^\bullet(\Omega_i, \tr) \to H\Lambda^\bullet(\Tilde{U}_i)$, for each $i \in \Ic^p$, the trace operator $\tr_{j}: H\Lambda^\bullet(\Omega_i, \tr) \to H\Lambda^\bullet(\Omega_{j}, \tr)$,  as well as the iterated trace $\tr^2_{0,1,2}: H\Lambda^\bullet(\Omega_i, \tr) \to H\Lambda^\bullet(\Omega_{0,1,2}, \tr)$ to construct the cochain maps. We will omit the subscript of the trace operators indicating the codomain of the operator.

Consider the following cochain map for functions on the top-level domains:
\begin{align}    
\Xi^{0,0}(f)_i = 
\begin{cases}
\phi^*_i (f_i) \qquad &\text{on} \; \Tilde U_i, \\
\phi_{j}^* (\tr \; f_i) \qquad  &\text{on} \; \Tilde U_{j}, \; j \in \Ic_i, \\
\phi_{0,1,2}^* (\tr^2 \; f_i) \qquad &\text{on} \; \Tilde U_{0,1,2}.
\end{cases}
\end{align}
The function $\Xi^{0,0}(f)_i \in H\Lambda^0(U_i)$, and $\Tilde{U}_i$, $\Tilde U_{j}$ and $\Tilde U_{0,1,2}$ are understood as subsets of $U_i$ and note that $\Tilde{U}_{0,1,2} = U_{0,1,2}$. The map $\phi_{j}^* (\tr \; f_i)$ takes the boundary values of the function $f_i \in H\Lambda^0(\Omega_i, \tr)$ and maps it to the domain $\Omega_{j}$. The boundary data is then mapped to $U_{j}$, extended constantly in the direction which is normal to $\Omega_{j}$. 

The next cochain map is $\Xi^1 = \Xi^{1,0} \oplus \Xi^{0,1}$. For functions on the interfaces $\Omega_{j}$, we have the following cochain map:
\begin{align}    
\Xi^{1,0}(g)_{j} = 
\begin{cases}
\phi_{j}^* (g_{j}) \qquad  &\text{on} \; \Tilde U_{j}, \\
\phi_{0,1,2}^* (\tr \; g_{j}) \qquad &\text{on} \; \Tilde U_{0,1,2}.
\end{cases}
\end{align}

The cochain map $\Xi^{0,1}$ is similar to $\Xi^{0,0}$, except we are considering the pullback of differential forms of degree $1$:
\begin{align}    
\Xi^{0,1}(\alpha)_{i} = 
\begin{cases}
\phi^*_i (\alpha_i) \qquad &\text{on} \; \Tilde U_i, \\
\phi_{j}^* (\tr \; \alpha_i) \qquad  &\text{on} \; \Tilde U_{j}, \; j \in \Ic_i, \\
0 \qquad &\text{on} \; \Tilde U_{0,1,2}.
\end{cases}
\end{align}
Notice that we have $\phi_{0,1,2}^* (\tr^2 \; \alpha_i) = 0$ because we are restricting a 1-form to a rank $0$ tangent bundle. This is a recurring theme, that some terms with iterative traces will vanish for higher degrees of the cochain map.

For the final cochain map we have the following decomposition: $\Xi^2 = \Xi^{2,0} \oplus \Xi^{1,1} \oplus \Xi^{0,2}$.
Since $\Omega_{0,1,2}$ is a single point, all functions are simply constant values. The map $\Xi^{2,0}$ is just a constant map onto $\Tilde U_{0,1,2}$:
\begin{equation}
\Xi^{2,0}(c) = \phi_{0,1,2}^* c = c|_{\Tilde U_{0,1,2}}.
\end{equation}    

We have a cochain map for 1-forms on the interfaces $\Omega_{j}$:
\begin{align}    
\Xi^{1,1}(\beta)_{j} = 
\begin{cases}
\phi_{j}^* (\beta_{j}) \qquad  &\text{on} \; \Tilde U_{j}, \\
0 \qquad &\text{on} \; \Tilde U_{0,1,2}.
\end{cases}
\end{align}

Following the same procedure, the 2-forms are evaluated to zero outside of $\Tilde{U}_i$, since we are restricting a 2-form to a rank 0 and rank 1 tangent bundle, respectively:
\begin{align}    
\Xi^{0,2}(\omega)_i = 
\begin{cases}
\phi^*_i \omega_i \qquad &\text{on} \; \Tilde U_i, \\
0 \qquad  &\text{on} \; \Tilde U_{j}, \; j \in \Ic_i, \\
0 \qquad &\text{on} \; \Tilde U_{0,1,2}.
\end{cases}
\end{align}

We will show in \cref{sub: the cochain maps} that the image $\Xi^{p,q}(\alpha)$ is weakly differentiable and that these maps satisfy the properties of a cochain map.

\subsection{Constructing an open cover with tubular neighborhoods} \label{subsection: constructing open cover}

Given a simplicial complex, we want to construct an open cover for the \v{C}ech-de Rham complex. We can accomplish this by considering subsets of tubular neighborhoods of the lower-dimensional simplices.

Let $\Omega$ be an $(n-p)$-dimensional submanifold of $\R^n$, e.g. a simplex with codimension $p$. For a given $x \in \Omega$, we say that a vector $n \in T_x \R^n$ is \emph{normal to $\Omega$} if $ \langle n, v \rangle =0$ for each $v \in T_x \Omega$, where $\langle \cdot , \cdot \rangle$ is the Euclidean inner product on $\R^n$. We define the \emph{normal space} of $\Omega$ at $x$ to be the following vector space:
\begin{equation}
N_x \Omega = \{n \in T_x \R^n : \langle n,v \rangle = 0, \forall v \in T_x \Omega \}.
\end{equation}
The \emph{normal bundle} is then the disjoint union of all the normal spaces:
\begin{equation}
    N\Omega = \coprod_{x \in \Omega} N_x \Omega.
\end{equation}

One can describe the normal bundle without referring to an underlying Euclidean space, where we instead consider a Riemannian manifold $M$ with a submanifold $\Omega$ and replacing the Euclidean inner product with a Riemannian metric $g$.

Moreover, this construction can be done more generally without the use of a Riemannian metric for smooth manifold, by defining the normal bundle as the quotient bundle $N\Omega = TM|_\Omega/T\Omega$. This defines the topology of the normal bundle.

We define the map $E: N\Omega \to \R^n$ by
\begin{equation}
E(x, v) = x + v.
\end{equation}
A \emph{tubular neighborhood} of $\Omega$ is a neighborhood $U_\Omega \subset \R^n$ containing $\Omega$, which is diffeomorphic under $E$ to 
\begin{equation}
V_\Omega = \{ (x,v) \in N\Omega : |v|< \epsilon(x) \},
\end{equation}
where $\epsilon: \Omega \to \R$ is a positive continuous function.
\begin{theorem}
Every submanifold $\Omega$ embedded into $\R^n$ has a tubular neighborhood. 
\end{theorem}
For proof of this claim, see \cite[Theorem 5.2]{hirsch2012differential} or \cite[Theorem 6.17]{lee2012smooth}. This means that we are able to assign a normal tubular neighborhood to each $(n-p)$-simplex, with $p$-dimensional normal space.

There are different ways to construct the open covers for a given simplicial geometry $\{\Omega_i\}_{i \in \Ic}$. One method for constructing an open cover is to assign a normal tubular neighborhood to each lower-dimensional simplex. Subsets of these tubular neighborhoods are chosen in such a way that they glue together to a simply connected open set $U_i$.

Another construction is to consider the open set as everything which is $\epsilon$-close to $\Omega_i$. For each $n$-simplex $\Omega_i$, we take a ball with radius $\epsilon$ centered at each point on the boundary $x \in \partial\Omega_i$. We define $U_i$ as the union of $\Omega_i$ and each of the open balls:
\begin{equation}
U_i = \Omega_i \cup \bigcup_{x \in \partial \Omega_i} B(x, \epsilon) .
\end{equation}
The two methods for constructing an open cover that are described above are shown in \cref{fig: two open cover constructions}. 

\begin{figure}[htb]
    \centering
   \includegraphics[scale=0.46]{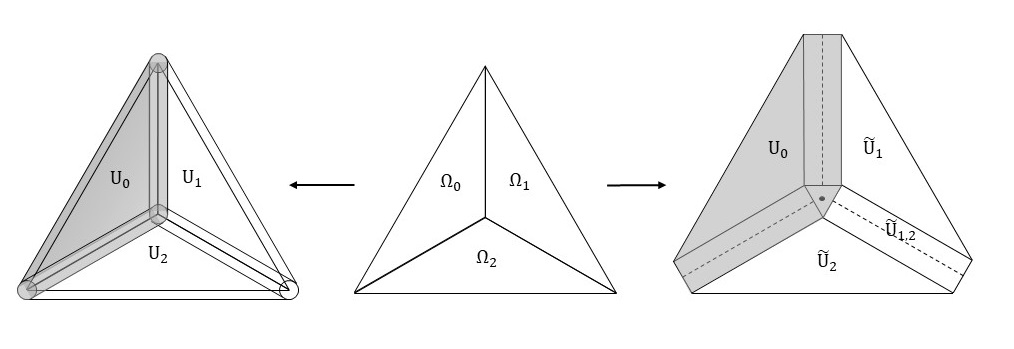}
    \caption{The figure illustrates two possible methods for constructing an open cover for a given simplicial geometry: assigning an $\epsilon$-radius (left), or "gluing" together subsets of tubular neighborhoods (right).}
    \label{fig: two open cover constructions}
\end{figure}

\textbf{The following summarizes the assumptions we put on the open cover $\Uc$}:
\begin{itemize}
    \item The indexing of the open cover $\Uc = \{U_i\}_{i \in \Ic}$ corresponds one-to-one with the indexing of $\{\Omega_i\}_{i \in \Ic}$.
    \item Each intersection $U_{i_0, ..., i_p}$ corresponds to a $(n-p)$-simplex $\Omega_{i_0, ..., i_p}$.
    \item The intersections $U_{i_0, ..., i_p}$ have a well-defined tangential direction and normal direction, as well as a positive distance function $\epsilon_i$.
    \item For $i \in \Ic^p$, we have that $\partial \Tilde{U}_i \cap \partial \Tilde{U}_j$ has zero $(n-1)$-dimensional measure for all $j \in \Ic_i^{p+s}$, where $s \geq 2$.
\end{itemize}
In the final assumption, the integer $s = |j| - |i|$ takes the role as the difference between the degree of overlap of $U_i$ and $U_j$. This assumption plays an important role in the next subsection, where we define the desired cochain map and show that the image of this cochain map is weakly differentiable. 
\subsection{The cochain maps}\label{sub: the cochain maps}
We now begin to describe the cochain map $\Xi^{p,q}$ locally for $i \in \Ic^p$. The cochain map is then obtained by assembly for each $i \in \Ic^p$.

We have the following formula for the cochain map:
\begin{align}\label{eq: cochain map} 
\Xi^{p,q}(\alpha)_i =   
\phi_{j}^* (\tr^{s} \; \alpha_i) \qquad &\text{on} \; \Tilde{U}_{j}, \qquad i \in \Ic^{p}, \; j \in \Ic_i^{p+s},
\end{align}
where $s = |j| - |i| \in \{0 ,... , n-p\}$ is the difference in degree of overlap between $U_i$ and $U_j$. We can write it explicitly for $\Tilde{U}_i$ and each $\Tilde{U}_{j} \subset \Tilde{U}_i$ in the following way:
\begin{align}  
\Xi^{p,q}(\alpha)_i = 
\begin{cases}   
\phi^*_{i} \alpha_i \qquad &\text{on} \; \Tilde U_{i}, \qquad i \in \Ic^p, \\
\vdots  \\
\phi_{j}^* (\tr^{s} \; \alpha_i) \qquad &\text{on} \; \Tilde{U}_{j}, \qquad j \in \Ic_i^{p+s}, \\
\vdots  \\
\phi_{m}^* (\tr^{n-p} \; \alpha_i) \qquad &\text{on} \; \Tilde{U}_{m},\qquad m \in \Ic_i^{n}. \\
\end{cases}
\end{align}

In order for us to prove the cochain property for $\Xi$, we first need to show that everything in the image of $\Xi$ permits a weak derivative.

\begin{proposition}\label{prop: weak derivative}
For any $\alpha \in \Sc^{p,q} = \prod_{i \in \Ic^p} H\Lambda^q(\Omega_i, \tr)$, we have that $\Xi^{p,q}(\alpha) \in \prod_{i \in \Ic^p} H\Lambda^q(U_i)$.  
\end{proposition}
\begin{proof}
We prove this for arbitrary fixed $p,q$ and $i \in \Ic^p$. We need to show that for a given $\alpha_i \in H\Lambda^q(\Omega_i, \tr)$, then $\hat{\alpha} = \Xi^{p,q}(\alpha_i)$ admits a weak derivative. We show this by proving that the right-hand side of the following equation exists:
\begin{equation}
\int_{U_i} d\hat{\alpha} \wedge \beta = (-1)^{q+1} \int_{U_i} \hat{\alpha} \wedge d\beta, \qquad \beta \in C^\infty_0\Lambda^{n-q-1}(U_i).
\end{equation}
In the expression above, $\beta$ is a smooth test-form which vanishes on the boundary of $U_i$, and it has degree $n-q-1$ so that $d\hat{\alpha} \wedge \beta$ and $\hat{\alpha} \wedge d\beta$ has degree $n$. Integration by parts and Stokes theorem gives us the following formula:
\begin{equation}
\int_{U_i} d\hat{\alpha} \wedge \beta = (-1)^{q+1} \int_{U_i}\hat{\alpha} \wedge d\beta + \int_{\partial U_i} \tr \: \hat{\alpha} \wedge \tr \: \beta.
\end{equation}

We write $\Ic_{i,+} = \Ic_i \cup \{i\}$, where $\Ic_i$ denotes all multi-indices containing $i$. If we show that 
\begin{equation}\label{eq: boundary integrals vanishes}
\sum_{j \in \Ic_{i,+}} \int_{\partial \Tilde{U}_j} \tr \: \hat{\alpha}_j \wedge \tr \: \beta_j = 0,
\end{equation}
then we can use the decomposition of $U_i$ given by \cref{eq: u tilde} to obtain the integration by parts formula:
\begin{equation}\label{eq: weak derivative of alpha}
\int_{U_i} d\hat{\alpha} \wedge \beta = (-1)^{q+1} \int_{U_i} \hat{\alpha} \wedge d\beta =  (-1)^{q+1} \sum_{j \in \Ic_{i,+}} \int_{\Tilde{U}_j} \hat{\alpha}_j \wedge d\beta_j = \sum_{j \in \Ic_{i,+}} \int_{\Tilde{U}_j} d\hat{\alpha}_j \wedge \beta_j.
\end{equation}
The first equality is the definition of weak derivative, the second equality is using the countable additivity of the Lebesgue integral, and the last equality holds if we show \cref{eq: boundary integrals vanishes}. 

We write $\Gamma_{i,j}$ for the boundary $\partial \Tilde{U}_i \cap \partial \Tilde{U}_j$, $j \in \Ic_i^{p+1}$, with orientation induced by the first index (each $\Tilde{U}_j$ is $n$-dimensional and inherit the standard orientation from $\R^n$). The test form $\beta$ vanishes on the boundary of $U_i$, hence we only need to account for the boundaries of $\Tilde{U}_j$ that are interior in $U_i$.

\begin{equation}\label{eq: int 2.3.1}
\int_{\partial \Tilde{U}_i} \tr \: \hat{\alpha} \wedge \tr \: \beta = \sum_{j \in \Ic_i^{p+1}} \int_{\Gamma_{i,j}} \tr \; \hat{\alpha} \wedge \tr \: \beta.
\end{equation}
For each integral in the right-hand side of \cref{eq: int 2.3.1}, there is a corresponding one from $\int_{\partial \Tilde{U}_j} \tr \: \hat{\alpha} \wedge \: \beta$ with opposite orientation due to different direction for the outward-pointing normal vector.

That is, for each integral from the boundary $\partial \Tilde{U}_i$,
\begin{equation}
I_{i,j} = \int_{\Gamma_{i,j}} \tr \: \hat{\alpha} \wedge \tr \: \beta = \int_{\Gamma_{i,j}} \tr \: \phi^*_i (\alpha_i) \wedge \tr \: \beta,
\end{equation}
there is a corresponding integral from $\partial \Tilde{U}_j$ with opposite orientation:
\begin{equation}\label{eq: int j i}
I_{j,i} = \int_{\Gamma_{j,i}} \tr \: \hat{\alpha} \wedge \tr \: \beta  = \int_{\Gamma_{j,i}} \tr \: \phi_j^* (\tr \: \alpha_i)  \wedge \tr \, \beta.
\end{equation}
By the definition of the pullback $\phi_j^*$, it extends the value of $\tr \; a$ constantly in the direction normal to the boundary $\Gamma_{i,j}$. The two trace operators in \cref{eq: int j i} applies the same restriction of the tangent bundle (a projection normal to the boundary $\Gamma_{i,j}$). The integrands are therefore equal, and since the integrals have opposite orientation, they cancel each other out.

More generally, we compare the integral
\begin{equation}
I_{l,m} = \int_{\Gamma_{l,m}} \tr \: \phi_l^* (\tr^r \alpha_i) \wedge \tr \: \beta
\end{equation}
with its counterpart
\begin{equation}
I_{m,l} = \int_{\Gamma_{m,l}} \tr \: \phi_m^* (\tr^{r+1} \alpha_i) \wedge \tr \: \beta.
\end{equation}
Here, $l \in \Ic^{p+r}$ and $m \in \Ic^{p+r+1}_l$, and the integrals $I_{l,m}$ and $I_{m,l}$ comes from
$\partial \Tilde{U}_l$ and $\partial \Tilde{U}_m$, respectively. The expression $\phi_m^* (\tr^{r+1} \alpha_i)$ is constant in $p+r+1$ directions, including the normal direction to the boundary $\Gamma_{m,l}$. On the other hand, $\phi_l^*(\tr^r \; \alpha_i)$ is not constant normal to $\Gamma_{l,m}$. The two expressions inside the integrals have therefore the same codimension and are equal.

We conclude that the integrals $I_{l,m}$ and $I_{m,l}$ satisfies $I_{l,m} = - I_{m,l}$ for every $l \in \Ic^{p+r}_{i, +}$, $m \in \Ic^{p+r+1}_l$, and we have satisfied \cref{eq: boundary integrals vanishes}. As a consequence, \cref{eq: weak derivative of alpha} holds, and hence we have shown that the image of the cochain map permits a weak derivative.
\end{proof}

The fact that all the interior boundary integrals vanish plays an important role in the next proposition.
\begin{proposition}
\cref{eq: cochain map} describes an injective bigraded cochain map from the simplicial de Rham complex to the truncated \v{C}ech-de Rham complex. That is, for each $p,q \geq 0$ satisfying $p+q \leq n$, $\Xi^{p,q}: \Sc^{p,q} \to \Ac^{p,q}$ satisfies
\begin{equation}
\Xi^{p,q+1} d_\Sc = d_\Ac \Xi^{p,q}, \qquad \Xi^{p+1, q} \delta_\Sc = \delta_\Ac \Xi^{p,q}.
\end{equation}
\end{proposition}

\begin{proof}
We let $\alpha = (\alpha_i)_{i \in \Ic^p} \in \Sc^{p,q}$ be an arbitrary differential form of degree $(p, q)$. First, we want to show that the cochain map $\Xi^{p,q}$ commutes with the exterior derivative, for an arbitrary choice of $p$ and $q$. That is, $d \Xi^{p,q}(\alpha) = \Xi^{p,q+1} (d\alpha)$.

On one hand, applying the cochain map followed by the exterior derivative yields the following:
\begin{align}   
d\Xi^{p,q}(\alpha)_i = 
\begin{cases}   
d\phi^*_{i} \alpha_i \qquad &\text{on} \; \Tilde U_{i}, \qquad i \in \Ic^p, \\
\vdots  \\
d\phi_{j}^* (\tr^{s} \; \alpha_i) \qquad &\text{on} \; \Tilde{U}_{j}, \; \qquad j \in \Ic_i^{p+s}, \\
\vdots  \\
d\phi_{m}^* (\tr^{n-p} \; \alpha_i) \qquad &\text{on} \; U_{m},\qquad m \in \Ic_i^{n}. \\
\end{cases}
\end{align}

On the other hand, 
\begin{align}   
\Xi^{p,q}(d\alpha)_i = 
\begin{cases}   
\phi^*_{i} d\alpha_i \qquad &\text{on} \; \Tilde U_{i}, \qquad i \in \Ic^p, \\
\vdots  \\
\phi_{j}^* (\tr^{s} \; d\alpha_i) \qquad &\text{on} \; \Tilde{U}_{j}, \; \qquad j \in \Ic_i^{p+s}, \\
\vdots  \\
\phi_{m}^* (\tr^{n-p} \; d\alpha_i) \qquad &\text{on} \; U_{m}, \; \qquad  m \in \Ic_i^{n}. \\
\end{cases}
\end{align}
We have shown in \cref{prop: weak derivative} that due to cancellation of terms, the interior boundaries don't contribute when we consider the decomposition of $U_i$. We have equality between the expressions since the exterior derivative commutes with pullbacks, and the trace operators are the pullback maps of inclusions, hence also commute with the exterior derivative. 

Secondly, we want to show that the cochain map commutes with the difference and jump operator, i.e. that  $\delta_{\Ac} \Xi^{p,q}(\alpha) = \Xi^{p+1, q} (\delta_\Sc \alpha)$. Recall that the difference operator is given by the following:
\begin{align}
(\delta \alpha)_i &=  \sum_{l=0}^{p+1} (-1)^{k+l}  \; \alpha_{i_0, ..., \hat{i}_l, ..., i_{p+1}}|_{U_i}, &
\forall i \in \mathcal{I}^{p+1}.
\end{align}
The jump operator $\delta_\Sc$ is structurally similar to the difference operator $\delta_\Ac$, where the dissimilarity is that we take the trace onto boundaries instead of restricting to overlaps:
\begin{align}
(\delta_\Sc \alpha)_i &= \sum_{l=0}^{p+1} (-1)^{k+l} \tr \; \alpha_{i_0, ..., \hat{i}_l, ..., i_{p+1}}, &
\forall i \in \mathcal{I}^{p+1}.
\end{align}

We let $i \in \Ic^{p+1}$ and $j \in \Ic_i^{p+s}$ (where $s$ takes the same role as in \cref{eq: cochain map}) and write $i \setminus l$ as short-hand for $(i_0, ..., \hat{i}_l, ..., i_{p+1}) \in \Ic^{p}_i$. By applying $\Xi$ first, then the difference $\delta_\Ac$ is defined locally on $\Tilde{U}_{i}$ by
\begin{align}   
(\delta_\Ac \Xi^{p, q}(\alpha))_i = 
\sum_{l=0}^{p+1} (-1)^{k+l}  (\Xi^{p,q}(\alpha))_{i \setminus l}|_{U_i} =  \sum_{l=0}^{p+1} (-1)^{k+l} \phi^*_i \tr (\alpha_{i \setminus l}), \qquad \text{on} \; \Tilde{U}_{i}.
\end{align}
More generally on $\Tilde{U}_{j}$, we have
\begin{align}   
\delta_\Ac (\Xi^{p, q}(\alpha))_i 
 = \sum_{l=0}^{p+1} (-1)^{k+l}  \phi^*_{j} \tr^{s+1} (\alpha_{i \setminus l}), \qquad \text{on} \; \Tilde{U}_{j}.
\end{align}

If we employ the jump operator $\delta_\Sc$ first, followed by the cochain map, we have the following:
\begin{align}
\Xi^{p+1,q}(\delta_\Sc \alpha)_i = \phi_{j}^* \tr^s (\delta_\Sc \alpha)_i  = \phi_{j}^* \tr^s \left(\sum_{l=0}^{p+1} (-1)^{k+l} \tr \; \alpha_{i \setminus l} \right)_i, \qquad \text{on} \; \Tilde{U}_{j}.
\end{align}
Because the pullback (and hence also the trace) is a linear operator, it commutes with the finite sum and hence we have equality between the two expressions. We therefore conclude that the cochain maps commute with the jump/difference operator. 

Finally, we verify that the image of the cochain map $\Xi^n$ is a subset of $\ker D^n$. We let $\omega \in \Sc^{p,q}$ for $p+q=n$, and consider the image $\Xi^{p,q}(\omega)$. If $\omega \in \Sc^{0,n}$, then clearly the exterior derivative of $\omega$ vanishes, since we cannot have differential forms of degree $n+1$ in $n$-dimensional space. If $\omega \in \Sc^{p,q}$, with $p+q=n$ and $q<n$, then $\omega_i$ is a volume form on $\Omega_i$. Hence, $\omega_i$ and $\Xi^{p,q}(\omega_i)$ are $q$-forms with just one component. The differential form is extended constantly in the normal directions given by the normal tubular neighborhood, i.e. the function component of $\omega$ is constant in $p$ variables. We have the following expression for the exterior derivative of $\omega_i$ using local coordinates:
\begin{equation}
d\omega_i = \sum_{l=1}^n \frac{\partial w_i}{\partial x_l} dx_l \wedge dx_i.
\end{equation}
For each $l \in \{1, ..., n\}$, either the partial derivative $\frac{\partial w_i}{\partial x_l}$ is zero because $w_i$ is constant in the normal direction, or the covector $dx_l$ is already in $dx_i$. Each summand is therefore zero, and hence the exterior derivative of any differential form in the image of $\Xi^n$ must be zero.

We also require that everything in the image $\Xi^n(\omega)$ is mapped to zero by the difference operator $\delta_\Ac$. For $\omega \in \Sc^{p,q}$, the extension $\Xi^{p,q}(\omega)_i$ is nonzero only on $\Tilde{U}_i$, and vanishes on overlaps of higher degree:

\begin{align}
\Xi^{p,q}(\omega)_i = 
\begin{cases}   
\phi^*_{i} \omega_i \qquad &\text{on} \; \Tilde U_{i}, \qquad i \in \Ic^p, \\
\vdots  \\
\phi_{j}^* (\tr^{s} \; \omega_i) = 0 \qquad &\text{on} \; \Tilde{U}_{j}, \qquad  j \in \Ic_i^{p+s}, \\
\vdots  \\
\phi_{m}^* (\tr^{n-p} \; \omega_i) = 0 \qquad &\text{on} \; U_{m}, \qquad m \in \Ic_i^{n}. \\
\end{cases}
\end{align}
Since $\omega_i$ is a volume form on $\Omega_i$, its trace is zero, and hence $\Xi^{n}(\omega)_i$ is only nonzero on $\Tilde U_{i}$. When we apply $\delta_\Ac$ we get that each $\Xi^{n}(\omega)_i$ is zero on overlaps, and hence the image of $\delta_\Ac^n$ is zero. We therefore conclude that the cochain map described is compatible with the truncation of the \v{C}ech-de Rham complex.

We have shown that the cochain maps in \cref{eq: cochain map} commute with the exterior derivative, commute with the jump/difference operators and that the image of the last cochain map $\Xi^n$ is in $\ker D^n$. 
\end{proof}

\section{Boundedness of the cochain map} \label{section: properties of the cochain map}
We have in the previous section defined a cochain map from the simplicial de Rham complex to the \v{C}ech-de Rham complex, and we continue in this section by proving that $\Xi$ is bounded from both above and below. In order to do so we need to introduce the various norms associated with the two complexes. Since we are working with Hilbert complexes, each Hilbert complex has an associated inner product, and the associated inner product induces a norm $\| \alpha \| = \langle \alpha, \alpha \rangle_{C^k}^{1/2}$.

Recall that the graph norm of a cochain complex is defined as follows: 
\begin{equation}
\| \alpha \|_{D}^2 = \|\alpha\|^2 + \|D\alpha\|^2,
\end{equation}
where we have omitted the subscript for the $L^2$ norms. Locally, the graph norm of the simplicial de Rham complex takes the following form: 
\begin{equation}
\|\alpha_i\|_{\Sc^{p,q}_i}^2 = \|\alpha_i\|^2 + \|d\alpha_i\|^2 +  \sum_{j \in \Ic_i}  \|\tr_j \alpha_i\|^2_{\Sc^{p+1,q}_j}. \label{eq: recursive sdr norm}
\end{equation}
We define a local graph norm for the \v{C}ech-de Rham complex to have the same recursive form as \cref{eq: recursive sdr norm}:
\begin{equation}
\|\beta_i\|_{\Ac^{p,q}_i}^2 = \|\beta_i\|^2 + \|d\beta_i\|^2 +  \sum_{j \in \Ic_i}  \|\beta_i|_{U_j}\|^2_{\Ac^{p+1,q}_j}. 
\end{equation}
As shorthand, we use subscripts $\Sc$ and $\Ac$ for the graph norms of the two double complexes and omit the indices.

\subsection{Proof of boundedness}
\begin{proposition}\label{prop: bounded 1}
$\Xi^{p,q}$ is a bounded cochain map, meaning there exists a constant $C_2$ (dependent on the thickness $\epsilon$) such that
\begin{equation}\label{eq: xi bounded above}
\|\Xi^{p,q}(\alpha) \|_{\Ac} \leq C_2 \| \alpha \|_{\Sc}, \qquad \forall \alpha \in \Sc^{p,q}.
\end{equation}
\end{proposition}
\begin{proof}
We consider the problem for a given component $\alpha_i$, $i \in \Ic^p$. If we show that there exists a constant such that 
\begin{equation}\label{eq: xi bounded above component}
 \|\Xi^{p,q}(\alpha_i) \|_{\Ac} \leq C_2 \|\alpha_i\|_{\Sc}, \qquad \forall \alpha_i \in H\Lambda^q(\Omega_i, \tr), 
\end{equation}
then the inequality in \cref{eq: xi bounded above} follows because we have a 1-1 correspondence between the domains $\Omega_i$ and $\Tilde{U}_i$. Using the definition of $\Xi^{p,q}$, we split $\|\Xi^{p,q}(\alpha_i)\|^2_\Ac$ further into the different components of the map $\Xi^{p,q}$:
\begin{align}
\|\Xi^{p,q}(\alpha_i)\|^2_{\Ac} = \| \phi_i^* \alpha_i \|^2_{\Ac} + \sum_{j \in \Ic_i} \|\phi^*_j \tr^s \alpha_i\|^2_{\Ac}.
\end{align}
Splitting $\|\Xi^{p,q}(\alpha_i)\|^2_\Ac$ into the sum of norms works because of \cref{prop: weak derivative}. From the constructions detailed in \cref{subsection: constructing open cover}, we assume that each $U_j$ is contained in an $\epsilon$-tubular neighborhood of $\Omega_j$. Since the pullbacks are constant in the normal directions, each pullback is bounded by a function of the length $\epsilon_j$:
\begin{equation} \label{eq: boundedness of pullback}
l(\epsilon_j) \|\alpha_i\|^2_{\Sc} \leq \|\phi_j^* \alpha_i \|^2_{\Ac} \leq L(\epsilon_j) \|\alpha_i\|^2_{\Sc},
\end{equation}

We therefore have
\begin{subequations}
\begin{align}
\|\Xi^{p,q}(\alpha_i)\|^2_{\Ac} &= \| \phi_i^* \alpha_i \|^2_{\Ac} + \sum_{j \in \Ic_i} \|\phi^*_j \tr^s \alpha_i\|^2_{\Ac} \\
&\leq L(\epsilon_i) \|\alpha_i\|^2_{\Sc}  + \sum_{j \in \Ic_i} L(\epsilon_j) \|\tr^s_j \alpha_i\|^2_{\Sc}.
\end{align}
\end{subequations}
Next, we make use of the norm in \cref{eq: recursive sdr norm}. The term $\|\tr_j^s \alpha_i \|^2_{\Sc}$ is contained in the definition of the norm $\|\alpha_i\|^2_{\Sc}$ due to its recursive definition, hence we have a bound:
 \begin{equation}
\|\tr_j^s \alpha_i \|^2_\Sc \leq \|\alpha_i\|^2_\Sc.
 \end{equation}
Putting it all together, we have the following estimate:
\begin{subequations}
\begin{align}
\|\Xi^{p,q}(\alpha_i)\|^2_{\Ac} &= \| \phi_i^* \alpha_i \|^2_{\Ac} + \sum_{j \in \Ic_i} \|\phi^*_j \tr^s_j \alpha_i\|^2_{\Ac} \\
&\leq L(\epsilon_i) \|\alpha_i\|^2_{\Sc}  + \sum_{j \in \Ic_i} L(\epsilon_j) \|\tr^s_j  \alpha_i\|^2_{\Sc} \label{eq: 2.36} \\
&\leq L(\epsilon_i) \|\alpha_i\|^2_{\Sc}  + \sum_{j \in \Ic_i} L(\epsilon_j)  \|\alpha_i\|^2_{\Sc} \\
&\leq (|\Ic_i|+1)\max_{j \in \Ic_i \cup \{i\}} L(\epsilon_j) \|\alpha_i\|^2_{\Sc}. \label{eq: final ineq}
\end{align}
\end{subequations}
The final inequality is obtained by choosing the largest constant, multiplied with the number of summands plus one  for the index $i$ itself. Lastly, by taking the square root we get \cref{eq: xi bounded above}.
\end{proof}

\begin{proposition}\label{prop: bound 2}
For each pair $(p,q)$ with $0 \leq p+q \leq n$, there exists a positive constant $C_1$ (dependent on $\epsilon$) such that
\begin{equation}\label{eq: xi bounded below}
  C_1 \|\alpha\|_{\Sc} \leq \|\Xi^{p,q}(\alpha)\|_{\Ac}, \qquad \forall \alpha \in \Sc^{p,q}.
\end{equation}
In particular, the cochain map $\Xi^{p,q}: \Sc^{p,q} \to \Ac^{p,q}$ is an injection.
\end{proposition}
\begin{proof}
Again, we break it down for a single component, as the result stated in the proposition then immediately follows. Taking the square of the right-hand side, we have
\begin{equation}
\|\Xi^{p,q}(\alpha_i) \|^2_{\Ac} = \| \phi_i^* \alpha_i \|^2_{\Ac} + \sum_{j \in \Ic_i} \|\phi_j^* \tr^s \alpha_i \|^2_{\Ac} \geq \|\phi_i^* \alpha_i\|^2_{\Ac}.
\end{equation}
Using the boundedness of the pullbacks from \cref{eq: boundedness of pullback}, we have
\begin{equation}
l(\epsilon_i) \|\alpha_i\|^2_{\Sc}  \leq  \|\phi_i^* \alpha_i \|^2_{\Ac} \leq \|\Xi^{p,q}(\alpha_i)\|^2_{\Ac}.
\end{equation}
Taking the square root completes the proof.
\end{proof}
Putting \cref{eq: xi bounded above} from \cref{prop: bounded 1} and \cref{eq: xi bounded below} from \cref{prop: bound 2} together, we get
\begin{equation} \label{eq: bounded above and below}
C_1 \|\alpha\|_\Sc \leq \|\Xi^{p,q}(\alpha)\|_\Ac \leq C_2 \|\alpha\|_\Sc.
\end{equation}

\begin{remark}
We emphasize that \cref{prop: bounded 1} and \cref{prop: bound 2} together imply that the cochain map is a bijection on its range, and that the range of $\Xi$ thus defines a subcomplex of $\Ac$ which is isomorphic to $\Sc$. Moreover, an isomorphism of cochain complexes implies an isomorphism in cohomology.
\end{remark}

We may choose to define an inner product and an associated graph norm that absorbs the constants $L(\epsilon_i)$. The local norm in \cref{eq: recursive sdr norm} is induced by the inner product:
\begin{align}    
\langle \alpha_i, \beta_i \rangle_{\Sc} &= \langle \alpha_i, \beta_i \rangle + \langle d\alpha_i, d\beta_i \rangle +  \sum_{j \in \Ic_i}  \langle \tr_j \alpha_i, \tr_j \beta_i \rangle_{\Sc}.
\end{align}
The length $L(\epsilon_i)$ can be added  as a weight to the inner product
\begin{equation}
\langle \alpha_i, \beta_i \rangle_{\Sc, \epsilon_i} =  \langle L(\epsilon_i)^2 \alpha_i, \beta_i \rangle_{\Sc}.
\end{equation} 
For constant weights $L(\epsilon_i)$, the associated weighted norm is given:
\begin{equation}
\|\alpha_i\|_{\Sc, \epsilon_i} = L(\epsilon_i) \|\alpha_i\|_{\Sc}.
\end{equation}

We can now rewrite \cref{eq: 2.36} using weighted inner products:
\begin{equation}
L(\epsilon_i) \|\alpha_i\|_{\Sc}  + \sum_{j \in \Ic_i} L(\epsilon_j) \|\tr^s_j  \alpha_i\|_{\Sc} = \| \alpha_i\|_{\Sc, \epsilon_i} + \sum_{j \in \Ic_i} \|\tr^s_j \alpha_i\|_{\Sc, \epsilon_j}.
\end{equation}
Following the same arguments as in \cref{prop: bounded 1}, we obtain a proof that $\Xi$ is bounded by a constant $C_1$, which is independent on our choice of $\epsilon_i$ for the open cover.
\section{Examples}
We consider the simplicial geometry $\mathbf{\Omega}=\{\Omega_i\}_{i \in \Ic}$ to be an equilateral triangle divided into three smaller triangles, denoted $\Omega_i$, $i=0,1,2$ with an open cover as depicted in \cref{fig: simplicial cech figure} and on the right side of \cref{fig: two open cover constructions}.

Throughout this section, we use $i$ to denote an index $i \in \Ic = \{0,1,2\}$, $j$ is a multi-index $\Ic^1$ and $l \in \Ic^2$.

Recall that for $\alpha \in \Sc^{0,0}$, we have the recursively defined norm:
\begin{subequations}
\begin{align}
\|\alpha\|^2_{\Sc} &= \sum_{i \in \Ic} (\|\alpha_i\|^2 + \|d \alpha_i\|^2 + \sum_{j \in \Ic_i} \| \tr_j \alpha_i \|^2_\Sc) \\
&= \sum_{i \in \Ic} \left(  \|\alpha_i\|^2_d + \sum_{j \in \Ic_i} \left( \|\tr_j \alpha_i\|^2_d +  \sum_{l \in \Ic_j} \|\tr^2_l \alpha_i\|^2_d \right)\right),
\end{align}
\end{subequations}

where the subscript $d$ denotes the graph norm of the de Rham complex, i.e. $\|\alpha\|^2_d = \|\alpha \|^2 + \|d\alpha \|^2$. 

For the specified geometry, we have an index set $\Ic = \{0,1,2\}$. We have three norms of codimension zero (one for each domain), six norms of codimension one (each $\alpha_i$ have two possible boundaries to be restricted to) and six norms of codimension two (each $\alpha_i$ restricted to $\Omega_{0,1,2}$ twice, once for each $\Omega_{i,j}$).

On the other hand, we have the recursive norm for the \v{C}ech-de Rham complex: 
\begin{subequations}
\begin{align}
\|\Xi(\alpha)\|^2_\Ac &= \sum_{i \in \Ic} \left( \|\Xi(\alpha_i)\|^2 + \|d\Xi(\alpha_i)\|^2 + \sum_{j \in \Ic_i} \|\Xi(\alpha_i)|_{U_{j}}\|_\Ac^2 \right) \\ 
&= \sum_{i \in \Ic} \left(\|\Xi(\alpha_i)\|_d^2  + \sum_{j \in \Ic_i} \left( \|\Xi(\alpha_i)|_{U_{j}}\|^2_d + \sum_{l \in \Ic_j}\|\Xi(\alpha_i)|_{U_{l}} \|_d^2 \right) \right).
\end{align}
\end{subequations}

Using the decomposition of $\Xi(\alpha)$, we get: 
\begin{equation}\label{eq: decomposition xialpha}
\|\Xi(\alpha)\|^2_{\Ac} = \sum_{i \in \Ic} \left(\| \phi^*_i \alpha_i \|^2_d + 2 \sum_{j \in \Ic_i} \| \phi_j^* \tr_j \alpha_i \|_d^2 + 5\|\phi_l^* \tr_l^2 \alpha_i\|_d^2 \right).
\end{equation}

Due to the overlapping nature of the \v{C}ech-de Rham complex and the choice of recursive norm, we get three norms on the non-overlapping subdomains $\Tilde{U}_i$, but a total of twelve norms for $\Tilde{U}_j$, and fifteen for the triple overlap $U_{0,1,2}$. 

The discrepancy between the number of terms in $\|\alpha\|_\Sc$ and $\|\Xi(\alpha)\|_\Ac$ can be explained by the fact that the lower-dimensional manifolds are counted several times in both cases, but they have measure zero in the mixed-dimensional setting. For example, the contributions of $\alpha_i$ on $\Omega_{j}$ is counted twice, once as the boundary of $\Omega_i$ (with measure zero) and once due to the recursive norm (with positive measure). When we look at the corresponding overlap $\Tilde{U}_j$, it appears twice with positive measure; once as a subset of $U_i$, and once due to the recursive norm.  

We introduce a norm for the simplicial de Rham complex which is weighted by the diameter $2\epsilon$:
\begin{subequations}
\begin{align}
    \|\alpha\|_{\Sc, \epsilon}^2 &= \sum_{i \in \Ic} \|\alpha\|_d^2 + \sum_{j \in \Ic_i} \| 2\epsilon \tr_j \alpha_i \|_{\Sc, \epsilon}^2 \\
&= \sum_{i \in \Ic} \left(\|\alpha_i\|_d^2  + \sum_{j \in \Ic_i} \left( \| 2\epsilon \tr_j \alpha_i\|^2_d + \sum_{l \in \Ic_j}\|4\epsilon^{2} \tr^2_l \alpha_i \|_d^2 \right) \right).
\end{align}
\end{subequations}
If the length of the sides of the larger triangle is equal to $L$, and the width of the extension is given equal to $2\epsilon$, then we have the following measures:

\begin{table}[ht!]
\centering
\begin{tabular}{||c c||} 
 \hline
 Mixed-dimensional measures & Equidimensional measures \\ [0.5ex] 
 \hline\hline
 $\mu(\Omega_i) = \frac{\sqrt{3}}{12}L^2$ &  $\mu(\Tilde{U}_i) = \frac{\sqrt{3}}{12}L^2$  \\ 
 $\mu(\Omega_{j}) = \frac{\sqrt{3}}{3}L$ & $\mu(\Tilde{U}_{j}) = \frac{2\sqrt{3}}{3}L\epsilon$  \\
 $\mu(\Omega_{0,1,2}) = 1$ & $\mu(\Tilde{U}_{0,1,2}) = \sqrt{3}\epsilon^2$  \\ [1ex] 
 \hline
\end{tabular}
\end{table}

\begin{example}    
The first and simplest example is when we let each $\alpha_i$ be a constant function on $\Omega_i$. Then $d\alpha_i = 0$, and the norm squared of each $\alpha_i$ is then 
\begin{equation}
\| \alpha_i\|^2 = \alpha_i^2 \mu(\Omega_i).
\end{equation}
 
The norm of $\alpha$ is then as follows:
\begin{equation}
\|\alpha\|_\Sc^2 =  \sum_{i=0}^2 \alpha_i^2 \left(\frac{\sqrt{3}}{12} L^2 +  \frac{2\sqrt{3}}{3}L + 2 \right).
\end{equation}
If we instead consider the weighted norm form $\alpha$, we have:
\begin{equation}\label{eq:Ex1a}
\|\alpha\|_{\Sc, \epsilon}^2 =  \sum_{i=0}^2 \alpha_i^2 \left(\frac{\sqrt{3}}{12} L^2 +  \frac{4\sqrt{3}}{3}L\epsilon + 8\epsilon^2 \right).
\end{equation}

On the other hand, the norm of $\Xi(\alpha)$ is given by:
\begin{equation}\label{eq:Ex1b}
\|\Xi(\alpha)\|_\Sc^2 =  \sum_{i=0}^2 \alpha_i^2  \left(\frac{\sqrt{3}}{12} L^2 +  \frac{8\sqrt{3}}{3}L\epsilon + 5\sqrt{3}\epsilon^2 \right).
\end{equation}

By comparing equations \eqref{eq:Ex1a} and \eqref{eq:Ex1b} we note that the leading order terms are identical, showing that the norms converge as $\epsilon \rightarrow 0$.
\end{example}

\begin{example}
Instead of limiting the problem to constant functions, we consider $\alpha_i \in H\Lambda^0(\Omega_i, \tr)$ and write $a_i = \|\alpha_i\|_d^2$. Similarly, we write $b_{j} = \| \tr_j \alpha_i\|_d^2$ and $c_{i} = \| \tr_l^2 \alpha_i\|_d^2$.

The norm for the simplicial de Rham complex is as follows:
\begin{equation}
\|\alpha\|_\Sc^2 = \sum_{i \in \Ic} \left( a_i + \sum_{j \in \Ic_i} b_{j} + 2c_i \right).
\end{equation}
Using the weighted inner products for the simplicial de Rham complex, we have the following formula for the norm:
\begin{equation}
\|\alpha\|_{\Sc, \epsilon}^2 = \sum_{i \in \Ic} \left(a_i +  2\epsilon \sum_{j \in \Ic_i} b_{j}  + 8\epsilon^2 c_i  \right)
\end{equation}

Since the image of the cochain map is constant in the normal direction(s) of the overlaps, $\| \phi_j^* \tr_j \alpha_i \|_d^2 = 2\epsilon b_{j}$. Similarly, $\|\phi_l^* \tr_l^2 \alpha_i\|_d^2 = \sqrt{3}\epsilon^2 c_i$. Together they lead to the following equation for the norm of $\Xi(\alpha)$:
\begin{equation}
\|\Xi(\alpha)\|_\Ac^2 = \sum_{i \in \Ic} \left(a_i + 4\epsilon \sum_{j \in \Ic_i} b_{j}  + 5\sqrt{3}\epsilon^2 c_i  \right)
\end{equation}

Constants $C_1 = 1$ and $C_2 = 2$ satisfies \cref{eq: bounded above and below} for this specific simplicial geometry and choice of open cover construction. 
\end{example}

\section{Summary and applications}
We consider a simplicial geometry and construct an associated open cover with a matching index set $\Ic$, where each lower-dimensional manifold $\Omega_{j}$ is given a thickness $\epsilon_j$ and corresponds to the overlap $U_j = U_{j_0} \cap ... \cap U_{j_p}$. We define the Hilbert simplicial-de Rham complex for the simplicial geometry and the Hilbert \v{C}ech-de Rham complex for the open cover. By considering the correct subcomplexes of the simplicial de Rham complex and the \v{C}ech-de Rham complex, we are able to construct a cochain map between the two complexes which is bounded from above and below. An explicit formula for the cochain map is expressed in \cref{eq: cochain map}. As a consequence, we can realize the simplicial de Rham complex as a subcomplex of the \v{C}ech-de Rham complex. 

Understanding that the simplicial de Rham complex can be realized as a subcomplex of the \v{C}ech-de Rham complex is an important tool for developing not just approximations between functions, but indeed, understanding the relationship between solutions to partial differential equations. In particular, we mention that the subcomplex property is critical for developing an \emph{a priori} theory for the solutions to partial differential equations posed on the two complexes \cite{boffi2013mixed, arnoldFEEC}.  Conversely, the subcomplex property is also a critical component in abstract \emph{a posteriori} results \cite{repin2008posteriori, pauly2020solution}. Simply said, we expect that the results of this paper can form a key ingredient in understanding the model approximation properties between large classes of mixed-dimensional and equidimensional models of compatible physical phenomena.  

\subsection*{Acknowledgments}
DFH and JMN acknowledge the support of the VISTA program, The Norwegian Academy of Science and Letters, and Equinor.

\bibliographystyle{plain}
\bibliography{References.bib}

\end{document}